\documentclass{amsart}
\usepackage{amsmath}
\usepackage{amssymb}
\usepackage{amsthm}
\usepackage{comment}
\newtheorem{theorem}{Theorem}[section]
\newtheorem{lemma}[theorem]{Lemma}
\newtheorem{conjecture}[theorem]{Conjecture}
\newtheorem{m-theorem}{Main theorem}
\newtheorem{problem}[theorem]{Problem}
\newtheorem{proposition}[theorem]{Proposition}

\newtheorem{corollary}[theorem]{Corollary}
\theoremstyle{remark}
\newtheorem{remark}[theorem]{Remark}

\newtheorem*{theorem*}{Theorem}
\newtheorem*{m-theorem*}{Main theorem}
\newtheorem*{lemma*}{Lemma}
\newtheorem*{proposition*}{Proposition}
\newtheorem*{definition*}{Definition}
\newtheorem*{corollary*}{Corollary}
\newtheorem*{remark*}{Remark}
\newtheorem*{problem*}{Problem}
\newtheorem*{conjecture*}{Conjecture}
\numberwithin{equation}{section}

\newcommand{\R}{\mathbb{R}}

\DeclareMathOperator{\dia}{diam}
\DeclareMathOperator{\Iso}{Isom}

\DeclareMathOperator{\past}{past}

\DeclareMathOperator{\Hess}{Hess}

\newcommand{\nb}[1]{(\mathrm{#1})}

\begin{document}

\title[Properly 
discontinuous  isometric 
group actions  ]{Properly 
discontinuous  isometric \\
group actions on 
inhomogeneous \\
Lorentzian manifolds}

\author{Jun-ichi Mukuno}
\address{Graduate School of Mathematics, 
Nagoya University, 
Chikusaku, Nagoya, 464-8602
Japan}
\email{m08043e@math.nagoya-u.ac.jp}



\date{\today}


\begin{abstract}
In the present paper, we prove that  no infinite group  
acts isometrically,  effectively, and properly discontinuously on a certain class of Lorentzian manifolds 
that are not necessarily homogeneous. 
\end{abstract}
\maketitle
\section{Introduction.}
Let us start with reminding ourselves what is called  
the existence problem of 
a compact Clifford--Klein form, in which 
$M$ is  a homogeneous space $G/H$ 
with $G$ being a Lie group and $H$ a closed subgroup of $G$. 
\begin{problem}
Does there exist a  subgroup $\Gamma$ of $G$ such  that 
the restricted action of $\Gamma$ on $M$  is properly discontinuous 
(namely,  that for any compact subset $K \subset M$, only finitely many  elements 
$\phi$  of $\Gamma$ 
satisfy $\phi (K) \cap K \neq \emptyset$), 
and that  
the  quotient manifold $\Gamma \backslash M$, called a 
Clifford--Klein form, 
is compact? 
\end{problem} 
Whether the answer to the problem is affirmative  
depends on the choice of the pair $(G,\,H)$. 
For instance,  if $G$ is semisimple, and if $H$ is a maximal compact subgroup of $G$, 
the quotient space $G/H$ carries the structure of a Riemannian symmetric space. 
Borel--Harish-Chandra~\cite{MR0147566},  Mostow--Tamagawa~\cite{MR0141672}, 
and  Borel~\cite{Borel}  
proved that \emph{the Riemannian symmetric space $G/H$ 
always admits compact  Clifford--Klein forms.}  
In the case where $G= O(n+1,\,1)$ and  $H=O(n,\,1)$, no compact 
Clifford--Klein form of the space $G/H$ exists, as 
the following theorem indicates.  
\begin{theorem}[Calabi--Markus~\cite{Ca}]\label{tm:Ca-Ma}
There is no infinite subgroup of the Lorentz group $O(n+1,\,1)$ whose 
 restricted action on the space $O(n+1,\,1)/O(n,\,1)$ is  properly discontinuous. 
\end{theorem} 
We should  remark that, although the  group $O(n+1,\,1)$ is enriched with co-compact lattices, 
none of them  acts properly discontinuously on  $O(n+1,\,1)/O(n,\,1)$ 
via the left action. 
By Wolf~\cite{Wo}, Kulkarni~\cite{Kul_2}, and Kobayashi~\cite{Kob_2}, 
the theorem of Calabi--Markus  
was  generalized to a certain class of homogeneous spaces.

We say that the \emph{Calabi--Markus phenomenon} occurs in a  pseudo-Riemannian manifold  
if no groups but finite ones can act  isometrically, effectively, and properly discontinuously on it. 
We should notice that  $O(n+1,\,1)/O(n,\,1)$ is the 
so-called de Sitter space, 
which is  a geodesically complete simply connected Lorentzian manifold 
of positive constant sectional curvature. 

Kobayashi~\cite{Kob_7} proposed the following conjecture: 
\begin{conjecture}\label{con:Kob}
Let $M$ be a geodesically complete pseudo-Riemannian manifold of signature $(p,q)$ with $p \geq q >0$. 
Suppose that we have  a positive lower  bound on the sectional curvature of $M$. 
Then, 
\begin{enumerate}
\item[$\nb{i}$] $M$ is never compact. 
\item[$\nb{ii}$] if  $p+q \geq 3$, the fundamental group of $M$ is always finite.
\end{enumerate}
\end{conjecture}
Due to Calabi--Markus~\cite{Ca}, Wolf~\cite{Wo}, and Kulkarni~\cite{Kul_2}, 
the conjecture is true in the case where $M$ has constant curvature. 
The conjecture is analogous to the   theorem of Myers  in Riemannian geometry,  
while it  may be interpreted to concern whether the Calabi-Markus phenomenon occurs in 
a class of pseudo-Riemannian manifolds  of variable curvature.  
However, KulKarni~\cite{MR522040}  showed  the following theorem: 
\begin{theorem}[KulKarni~\cite{MR522040}]
Let $M$ be  a connected pseudo-Riemannian manifold with an indefinite metric of dimension $n \geq 3 $. 
If the sectional curvature of $M$ is either bounded from above or below, 
then $M$ is of constant curvature. 
\end{theorem}
According to the theorem above, 
if a  pseudo-Riemannian manifold $M$  of dimension $n \geq 3$  
satisfies the assumptions of the conjecture, 
$M$ has  constant curvature. 
Hence we wish to weaken  the condition on the sectional curvature in the conjecture.   
Modifying the hypotheses of the conjecture, we obtain the following proposition: 
\begin{proposition}\label{pr_11}
Let $(M,\,h)$ be a time-orientable  non-spacelike geodesically complete globally hyperbolic 
Lorentzian manifold of dimension $n+1\geq 2 $.  
Suppose that 
we have a positive lower bound $\alpha ^{2}$ on 
the sectional curvature on  any indefinite 
linear subspace of dimension two in $T_{p}M$ for any  $p \in M$. 
Assume that there exists a closed spacelike hypersurface $S$ in $M$ 
satisfying the following conditions: 
\begin{enumerate}
\item[$\nb{i}$]  $S$ is totally geodesic;
\item[$\nb{ii}$]  every inextendible non-spacelike geodesic passes through $S$  exactly once.  
\end{enumerate}
For any $p \in M - S$, let $I_{S}(p)$ be the set consisting of the points in $S$  which 
a timelike geodesic joins to $p$. 
\begin{enumerate}
\item[$\nb{iii}$]  $I_{S}(p)$ is geodesically connected for all $p \in M - S$; 
\item[$\nb{iv}$]  there is no intersection of $I_{S}(p)$ and lightlike  geodesic rays from $p$. 
\end{enumerate}
Then the Calabi--Markus phenomenon occurs in $M$.  
\end{proposition}
The proposition  is a generalization of  the theorem of Calabi--Markus,  
since we easily find a hypersurface  fulfilling the requirements on $S$. 
Moreover  we can drop the condition~$\nb{iii}$ if $M$ is a surface, and therefore deform 
the metric of the de Sitter surface outside a neighborhood of the hypersurface 
with  the assumptions of the proposition satisfied.  
However, in general, such a hypersurface $S$ hardly exists. 
Avoiding this difficulty, 
we present a further extension,  given below, of the  proposition. 

Let $F$ be a connected closed manifold of dimension $n \geq 1$, and 
$\{g_{t}\}_{t \in \R}$ a smooth family of Riemannian metrics of $F$. 
Denote by $M^{1,n}$  the Lorentzian manifold 
$(\R \times F,\,h_{M^{1,n}}=-dt^{2}+g_{t})$. 
The  partial derivative  in the direction $t$ 
is denoted by $\partial_{t}$.  
For $t \neq 0$,  $\partial_{|t|}$ 
is defined to be  the ``signed'' partial  derivative $(t/|t|) \partial_{t}$.  
We always put the following condition on $M^{1,n}$: 
\begin{enumerate}
\item[$(\rm{H}$$)_{t_{0},c}$]  there exist  positive constants $t_{0}$ and $c$ such that 
$\partial_{|t|}g_{t}(X,\,X) \geq 2c g_{t}(X,\,X)$ 
 for any  vector field $X$ on $F$  and $|t| \geq t_{0}$.
\end{enumerate}
The main result of our paper is 
\begin{theorem}
The Calabi--Markus phenomenon occurs in $M^{1,n}$. 
\end{theorem}

In Section~\ref{la:la_2}, we investigate the behavior of 
spacelike geodesics in the ends of $M^{1,n}$.  
The results of  Section~\ref{la:la_2} are essential for the proof of the main theorem, 
which is  proved in  Section~\ref{la:la_3}. 
In Section~\ref{sec_11}, we give a simple proof of Proposition~\ref{pr_11} by using  
 the divergence theorem. 
\section{Spacelike Geodesics in $M^{1,n}$.}\label{la:la_2}
In this section, we concern ourselves with spacelike geodesics in  $M^{1,n}$. 

Let $\gamma  $ be a curve in  $M^{1,n}$, 
and $\gamma_{F}  $ the curve in $F$  obtained by projecting $\gamma  $ 
onto $F$.  
We denote by $\Dot{\gamma}(u)$ and $\Dot{\gamma_{F}}(u)$
the velocities of  $\gamma  $ and $\gamma_{F}  $ at time $u$, respectively. 
We write $\pi_{\R}$ for 
the natural projection of $M^{1,n}=\R \times F$ onto $\R$. 
Take a local coordinate system $(x^{0},\,x^{1},\,\ldots,x^{n})$ 
of $M^{1,n}$ 
such that $x^{0}=\pi_{\R}$, and 
that 
$(x^{1},\,x^{2},\,\ldots ,x^{n})$ comes from 
a local coordinate system of $F$. 
Set $\partial_{i}=\partial_{x^i}$, $\gamma^{i}(u)= x^{i} ( \gamma (u))$, 
$\Dot{\gamma^{i}}(u)=d\gamma^{i}(u)/du$, and  
$\Ddot{\gamma^{0}}(u)=d^{2}\gamma^{0}(u)/du^{2}$ for $0 \leq i \leq n$.  
The symbols  $h_{ij}$, $h^{ij}$, and $\Gamma^{k}_{i j}$ stand for    
$h_{M^{1,n}}(\partial_{i},\,\partial_{j})$, the $(i,\,j)$-th entry of the inverse matrix of 
$(h_{i j})_{0\leq i,\,j \leq n}$, and 
the Christoffel symbols for $0 \leq i,\,j,\,k \leq n$,  
respectively, where $h_{M^{1,n}}$ is the Lorentzian metric of $M^{1,n}$. 

In what follows, we make  estimates on the lengths of spacelike geodesics of $M^{1,n}$.   
\begin{lemma}\label{le: le_1}
The length of a spacelike geodesic $\gamma$ in $\pi_{\R}^{-1}(
(-\infty,$ $-t_{0}] \cup [t_{0},\,\infty)
)$ 
is less than $\pi / c$. 
\end{lemma}
\begin{proof}
We require  $\gamma  $ to be parametrized by arc length. 
For simplicity, let  the domain of $\gamma$ be  $[0, \, L]$. 
As  $\gamma  $ maintains  unit speed, 
\begin{equation}\label{eq: eq_1}
h_{M^{1,n}}(\Dot{\gamma}(u),\,\Dot{\gamma}(u)) 
= - |\Dot{\gamma^{0}}(u)|^2 
+ g_{\gamma^{0}(u)}(\Dot{\gamma_{F}}(u),\,
\Dot{\gamma_{F}}(u)) =1.
\end{equation}
We assume that $\gamma  $ stays in $\pi_{\R}^{-1}([t_{0},\,\infty))$. 
Here we should notice that 
\begin{align*}
\Gamma^{0}_{0 0}= \Gamma^{0}_{i 0}=\Gamma^{0}_{0 j} =0,\  
\Gamma^{0}_{i j}=\frac{1}{2}\partial_{0} h_{i j}\ (1 \leq i,\,j \leq n). 
\end{align*}
Therefore we obtain 
\begin{align}\label{eq: a_eq_1}
\Ddot{\gamma^{0}}(u)+ 
\sum_{0 \leq i,j \leq n} \Gamma^{0}_{i j}\Dot{\gamma^{i}}(u)\Dot{\gamma^{j}}(u) 
= \Ddot{\gamma^{0}}(u) + 
\frac{1}{2} (\partial_{t}g)_{\gamma^0 (u)}
(\Dot{\gamma_{F}}(u),\,\Dot{\gamma_{F}}(u)). 
\end{align}
Since $\gamma  $ is a geodesic, the left-hand side  is zero. 
By $(\rm{H}$$)_{t_{0},c}$, (\ref{eq: eq_1}), and (\ref{eq: a_eq_1}),  
\begin{align}\label{eq: eq_2}
\Ddot{\gamma^{0}}(u) =- 
\frac{1}{2} (\partial_{t}g)_{\gamma^0 (u)}
(\Dot{\gamma_{F}}(u),\,\Dot{\gamma_{F}}(u)) 
\leq -c g_{\gamma^0(u)}
(\Dot{\gamma_{F}}(u),\,\Dot{\gamma_{F}}(u))
= -c (1+ |{\Dot\gamma^{0}}(u)|^2). \nonumber
\end{align}
Putting $G(u)=\cot (c u) $, 
we check at once that 
$dG(u) /du  = -c (1+G(u)^{2})$. 
We can take a positive number $u_{0} < \pi / c$ such that 
$\Dot{\gamma^{0}}(0) = G(u_{0})$. 
Then  by  the Riccati comparison argument 
we have $\Dot{\gamma^{0}}(u) \leq G(u+u_{0})$ 
as long as $\gamma^{0}(u) \geq t_{0}$ (see~\cite{Kar}). 
It immediately follows  that 
\begin{equation*}
\gamma^{0}(u) \leq  \gamma^{0}(0) + \int_{0}^{u} \cot c(s+u_{0})\ ds 
\end{equation*}
whenever $u$ satisfies $\gamma^{0}(u) \geq t_{0}$. 
If $u$ goes to $\pi/c -u_0$, the right-hand side approaches $-\infty$. 
As $\gamma^{0}(u) \geq t_{0}$, 
the length  of $\gamma  $ is less than $\pi/c -u_0< \pi/c$. 
The same argument applies to the case $\gamma \subset \pi_{\R}^{-1}(
(-\infty,\,-t_{0}])$ as well. 
\end{proof}
Write  $F_{t}$ for 
the spacelike submanifold $\pi_{\R}^{-1}(\{t\})$ in $M^{1,n}$, 
which is isometric to the Riemannian manifold $(F,\,g_t)$. 
Let $\pi_{t}$ be the natural projection of $M^{1,n}=\R \times F$ onto $F_{t}$. 
\begin{lemma}\label{le: le_2}
Let $\gamma : [0,\,L] \rightarrow M^{1,n}$  be a spacelike geodesic 
in $\pi_{\R}^{-1}((-\infty,\, -t_{0}] \cup [ t_{0},\, \infty))$ 
such that $\Dot{\gamma^{0}}(0)=0$. 
Then there is a constant $C^{\prime}$ that dominates the length of the spacelike curve
$\pi_{\gamma^{0}(L)}(\gamma  )$. 
\end{lemma}
\begin{proof}  
We give the proof in the case where $\gamma \subset
\pi_{\R}^{-1}( [t_{0},\,\infty))$, 
for the same proof works in the other case as well.   
We assume that $\gamma  $ is parametrized by arc length. 
Then we  see that 
\begin{equation*}
1+|\Dot{\gamma^{0}}(u)|^{2}=g_{\gamma^{0}(u)}
(\Dot{\gamma_{F}}(u),\,\Dot{\gamma_{F}}(u)). 
\end{equation*}
By $(\rm{H}$$)_{t_{0},c}$, for any vector field $X$ on $F$ and positive numbers
$t_{1},\,t_{2}$ with $t_{2} \geq t_{1} \geq t_{0}$, 
we have $g_{t_{2}}(X,\,X) \geq e^{2c (t_{2}-t_{1})} g_{t_{1}}(X,\,X)$. 
Since $\Dot{\gamma^{0}}(0)=0$, 
and $\gamma^{0}  $ is strictly concave, $\gamma^{0}  $ is strictly decreasing. 
It follows that 
\begin{align*}
\sqrt{g_{\gamma^{0}(L)}
(\Dot{\gamma_{F}}(u),\,\Dot{\gamma_{F}}(u))}
&=\sqrt{(1+|\Dot{\gamma^{0}}(u)|^{2})\left(\frac{g_{\gamma^{0}(u)}
(\Dot{\gamma_{F}}(u),\,\Dot{\gamma_{F}}(u))}{g_{\gamma^{0}(L)}
(\Dot{\gamma_{F}}(u),\,\Dot{\gamma_{F}}(u))}\right)^{-1}} \\
&\leq \left(\sqrt{1+|\Dot{\gamma^{0}}(u)|^{2}} \right) e^{-c (\gamma^{0}(u)-\gamma^{0}(L))} \\
&\leq (1+|\Dot{\gamma^{0}}(u)|) e^{-c (\gamma^{0}(u)-\gamma^{0}(L))}. 
\end{align*}
Integrating the above inequality, we find that %
\begin{align}\label{in_1}
&\int_{0}^{L} \sqrt{g_{\gamma^{0}(L)}
(\Dot{\gamma_{F}}(u),\,\Dot{\gamma_{F}}(u))} \ du \\
& \leq 
\int_{0}^{L} e^{-c (\gamma^{0}(u)-\gamma^{0}(L))} du + 
\int_{0}^{L} |\Dot{\gamma^{0}}(u)|\ e^{-c (\gamma^{0}(u)-\gamma^{0}(L))}\ du. \nonumber
\end{align}
As $e^{-c (\gamma^{0}(u)-\gamma^{0}(L))} \leq 1$, 
we see that  
\begin{equation*}
\int_{0}^{L} e^{-c (\gamma^{0}(u)-\gamma^{0}(L))} du \leq L.
\end{equation*}
Due to Lemma~\ref{le: le_1}, $L$ is bounded above by $\pi/c$. 
We estimate the second term of the right-hand side of the inequality~(\ref{in_1}). 
\begin{align*}
\int_{0}^{L} |\Dot{\gamma^{0}}(u)|\ e^{-c (\gamma^{0}(u)-\gamma^{0}(L))}\ du 
&= 
- \int_{0}^{L} \Dot{\gamma^{0}}(u)\ e^{-c (\gamma^{0}(u)-\gamma^{0}(L))}\ du \\
&=
- \int_{\gamma^{0}(0)}^{\gamma^{0}(L)} 
 e^{-c (s-\gamma^{0}(L))}\ ds 
\leq \frac{1}{c}.
\end{align*}
The proof is complete. 
\end{proof}
\begin{remark}\label{re: re_2}
Fix $T\geq t_{0}$. 
Let $\gamma : [0,\,1] \rightarrow \pi_{\R}^{-1}([T,\,\infty))$ be a spacelike 
geodesic. 
We extend $\gamma   $ as long as  $\gamma \subset \pi_{\R}^{-1}([T,\,\infty))$. 
Let $\overline{\gamma}   $ be the maximal  extension of $\gamma  $, 
and $I$ the domain of $\overline{\gamma}  $. 
We prove that each of the endpoints of $\overline{\gamma}  $ reaches $F_{T}$. 
By Lemma~\ref{le: le_1}, 
$I$ is a bounded interval.  
Suppose that 
 an endpoint of $I$ does not belong to $I$. 
For simplicity, we  put $I=[0,\,L)$, where $L>0$. 
We show that, if $u$ approaches $L$, $\overline{\gamma}(u)$ converges.     
As ${\overline{\gamma}^{0}}  $ is concave and bounded, the limit 
$\displaystyle\lim_{u \rightarrow L}  {\overline{\gamma}^{0}}(u)$ exists.
We regard $\overline{\gamma}_{F}  $ as a curve in $(F,\,g_{t_0})$. 
By Lemma~\ref{le: le_2}, the length of ${\overline{\gamma}}_F  $ is finite. 
Since $F$ is compact,  the limit $\displaystyle\lim_{u \rightarrow L} {\overline{\gamma}}_{F}(u)$ lies in $F$. 
Therefore we can extend $\overline{\gamma}  $ until the endpoint continuously. 
We say that an open set $U$ of a Lorentzian manifold $M$    is \emph{convex} 
if for any  point $p$  of $U$ there exists an open set $V$ of $T_{p}M$  such that  
the restriction of $\exp_{p}$ to $V$ is a diffeomorphism onto $U$, 
and that $v \in V$ implies $t v \in V$  for any $t \in [0,\, 1]$. 
The existence of   a convex neighborhood of any point of $M^{1,n}$ leads us to  
the extension of the geodesic $\gamma$  to the limit $\displaystyle\lim_{u \rightarrow L}  {\overline{\gamma}}(u)$. 
This contradicts our assumption. 
It follows that $I$ is closed. 
The maximality of $\overline{\gamma}  $ completes the proof. 
\end{remark}
Let $d_{T}$ be 
the ``intrinsic''  distance on $F_{T}$ defined by the Riemannian metric of $F_{T}$. 
Combining the preceding lemmas, we  obtain
\begin{corollary}\label{co: co_3}
Fix a positive constant $T>t_{0}$, and take
a connected spacelike geodesic $\gamma$  
in either $ \pi_{\R}^{-1}([T,\, \infty))$ or $\pi_{\R}^{-1}((-\infty,\,-T])$. 
Let $y,\,z$ be the endpoints of the geodesic $\gamma$. 
If $\gamma$ is included in $ \pi_{\R}^{-1}([T,\, \infty))$ (resp.\  in $\pi_{\R}^{-1}((-\infty,\,-T])$), 
then $d_{T}(\pi_{T}(y),\,\pi_{T}(z))$ (resp.\ $d_{-T}(\pi_{-T}(y),\,\pi_{-T}(z))$) 
is dominated by $2C^{\prime}$, where $C^{\prime}$ is the constant  in Lemma~\ref{le: le_2}. 
\end{corollary}
\begin{proof}
It is sufficient to show the case where $\gamma \subset 
\pi_{\R}^{-1}([T,\, \infty))$. 
 By  Remark~\ref{re: re_2}, 
we can extend the geodesic $\gamma$ until each of 
the endpoints of the geodesic reaches $F_{T}$. 
We  write this extension of $\gamma$ as $\overline{\gamma}$. 
$L(-)$ stands for the length of a spacelike curve. 
We have 
\begin{align*}
d_{T}(\pi_{T}(y),\,\pi_{T}(z)) \leq L(\pi_{T}(\gamma))
\leq L(\pi_{T}(\overline{\gamma})). 
\end{align*}
Let $[-L_{0},\,L_{1}]$ be the domain of $\overline{\gamma}$ such that 
$\Dot{\overline{\gamma}^{0}}(0)=0$.  
Then 
by Lemma~\ref{le: le_2} each of $L(\pi_{T}(\overline{\gamma}|_{[-L_{0},\,0]}))$ and 
$L(\pi_{T}(\overline{\gamma}|_{[0,\,L_{1}]}))$ 
is bounded above by $C^{\prime}$. 
\end{proof}
\section{Proof of The Main Result.}\label{la:la_3}
This section is devoted to the proof of the main result.

We denote by $\dia (-)$  the diameter of a metric space. First we show
\begin{lemma}\label{le: le_3.5}
Let $(X,\,d)$ be a connected compact metric space. 
For  any  family $\{A_{i}\}_{i=1}^{k}$ of connected closed subsets of $X$ 
such that $X=\displaystyle\bigcup_{i=1}^{k} A_{i}$,  
we have  
\begin{equation*}\label{eq: eq_a}
\dia (X) \leq \sum_{i=1}^{k} \dia (A_{i}).
\end{equation*}
\end{lemma}
\begin{proof}
As $X$ is compact, there are two points $y,\,z\in X$ satisfying $d(y,\,z)=\dia (X)$. 
Replacing the indices properly, we may require $A_1$ to contain  $y$. 
We have only to consider the case where $z$ belongs to some $A_i$ with $i \neq 1$. 
Rearranging $A_2,\, \ldots,\, A_{k}$ in an appropriate manner allows us to assume that  
$z\in A_{l}$ for some $l$, and that $A_{i} \cap A_{i+1} \neq \emptyset$ for any positive integer $i \leq l-1$. 
Put $w_1 = y$ and $w_{l+1} = z$, and take $w_{i+1} \in A_{i} \cap A_{i+1},\, i=1,\, 2 ,\, \ldots,\, l-1$.
Then we obtain  
\begin{align*}
 \dia (X) 
=d(y,\,z) &\leq  \sum_{i=1}^{l} d(w_i ,\,w_{i+1}) 
\leq \sum_{i=1}^{l} \dia (A_{i}) \leq \sum_{i=1}^{k} \dia (A_{i}). 
\end{align*}
\end{proof}
We  return to our manifold $M^{1,\,n}$.  
Next we  prove 
\begin{lemma}\label{le: le_4}
For any positive numbers $T$ and $\epsilon$, 
there are subsets $V_{1},\,V_{2},\, \ldots, \,
V_{n(T,\,\epsilon)}$ of  $F_{T}$  
fulfilling the following conditions: 
\begin{enumerate}
\item[\rm{(i)}] $V_{i},\,i=1,\,2,\, \ldots,\, 
n(T,\,\epsilon),$  are  homeomorphic to  a closed ball of  
dimension $n$;
\item[\rm{(ii)}] for any $p,\,q \in V_{i}$, $p$ is joined to $q$ by a spacelike geodesic in 
$\pi_{\R}^{-1}((T-\epsilon,\,T+\epsilon))$;  
\item[\rm{(iii)}] $F_{T}= \displaystyle\bigcup_{i=1}^{n(T,\,\epsilon)} V_{i}$. 
\end{enumerate} 
\end{lemma}
\begin{proof}
For any $x \in F_{T}$, we can take a convex neighborhood $W_{x}$ 
of $x$ in $\pi_{\R}^{-1}((T-\epsilon,\,T+\epsilon))$.
There is an embedding $\iota_{x}$ of a closed unit ball of dimension $n$ 
into 
$W_{x} \cap F_{T} $ such that 
$\iota_{x}(0)=x $. 
Let $V_{x}$  be its image.  
Since $V_{x}$ is included in  $W_{x}$, 
for any two points $p, \,q \in V_{x}$ 
we can find a geodesic from $p$ to $q$. 
Indeed the geodesic is spacelike. 
This follows from the fact that $\pi_{\R}(p)=\pi_{\R}(q)$, and that if  $\gamma$ is a 
non-spacelike geodesic, 
$\gamma^{0}$ is a strictly monotonic function. 
 Compactness of $F_{T}$ implies that  there are $x_{1},\, x_{2},\, \ldots ,\,x_{n(T,\,\epsilon)}
$ such that $F_{T}= \displaystyle\bigcup_{i=1}^{n(T,\,\epsilon)} V_{x_{i}}$. 
This completes  
the proof. 
\end{proof}
Let us prove the main theorem 
stated in the introduction.  
We denote by $\Iso (M^{1,n})$  the isometry group of $M^{1,n}$. 
We write $\Iso^{+}(M^{1,n})$ for the subgroup of $\Iso (M^{1,n})$ consisting of those 
elements that
preserve time-orientation. 
The Calabi--Markus phenomenon emerges if a certain compact set $K$ meets  
the image of $K$ under any $\phi \in \Iso (M^{1,n})$. 
The last claim remains valid even when 
 $\Iso (M^{1,n})$ is replaced by $\Iso^{+} (M^{1,n})$, 
since the index of $\Iso^{+} (M^{1,n})$ in $\Iso (M^{1,n})$ is at most two. 
Therefore it suffices to show the following lemma: 
\begin{lemma}
For some sufficiently large $T_2 >0$, 
 $K=\pi_{\R}^{-1}([-T_{2},\,T_{2}])$ fulfills  
$\phi (K) \cap K \neq \emptyset$ 
for any $\phi \in \Iso^{+} (M^{1,n}) $. 
\end{lemma}
\begin{proof}  
Take $T_{1} > t_{0}$ arbitrarily, and  $\epsilon >0$ so that 
$T_{1}- \epsilon >t_{0}$.  
  As  $(\rm{H}$$)_{t_{0},c}$ implies that 
$\dia (F_{T})$ has  an  at least 
 exponential growth  with respect to $T \geq t_{0}$,  
 we can find $T_{2} > T_{1}$ in such 
that $\dia (F_{T_{2}}) >  2 n(T_{1},\,\epsilon) C^{\prime}$, 
where  $n(T_{1},\,\epsilon)$ 
 and $C^{\prime}$ are the  constants appearing in 
Lemma~\ref{le: le_4} and Lemma~\ref{le: le_2}, respectively.

Suppose, on the contrary,  that 
\begin{equation*}
\phi (\pi_{\R}^{-1}([-T_{2},\,T_{2}])) \cap 
\pi_{\R}^{-1}([-T_{2},\,T_{2}]) = \emptyset
\end{equation*} 
for some isometry $\phi\in \Iso^{+} (M^{1,n})$. 
We have $\phi (F_{T_{1}}) \subset \pi_{\R}^{-1}((-\infty,\,-T_{2}] \cup [T_{2},\,\infty))$ 
as $T_{1} \in (0,\,T_{2})$. 
We consider only the case where $\phi (F_{T_{1}}) \subset 
\pi_{\R}^{-1}([T_{2},\,\infty))$, 
for the 
same argument applies to the remaining case as well. 
We should remark that, for any timelike geodesic  $\gamma:\R \rightarrow M^{1,\,n}$,     
$\gamma^{0}  $ is surjective. 
It follows that,  
for all $p \in F$, $\pi_{\R} (\phi^{-1} ({\pi_{T_{2}}}^{-1}(\{p\}))) =\R$, 
since ${\pi_{T_{2}}}^{-1}(\{p\})$ is totally geodesic. 
Hence we obtain $\pi_{T_{2}}(\phi (F_{T_{1}}))=F_{T_{2}}$.

On the other hand,    
we should notice that 
there exist $V_{i},\,i=1, 2, \ldots ,$ $n(T_{1},\,\epsilon),$ satisfying those conditions in
Lemma~\ref{le: le_4} with $T$ replaced by $T_{1}$.  
For each $i$, 
compactness of $\pi_{T_{2}}(\phi (V_{i}))$ guarantees the existence of 
  $y_{i},\,z_{i} \in V_{i}$ with 
\begin{equation*}
\dia (\pi_{T_{2}}(\phi (V_{i})))
= 
d_{T_{2}}(\pi_{T_{2}}( \phi (y_{i})),\,\pi_{T_{2}}(\phi (z_{i}))). 
\end{equation*}
By the choice of $\epsilon$, we can find a spacelike geodesic $\gamma_{i}$ from $y_{i}$ to $z_{i}$ in  
$\pi_{\R}^{-1}((t_{0},\,\infty))$. 
Since ${\gamma_{i}}^{0}  $ is concave, we have $\gamma_{i} \subset \pi_{\R}^{-1}([T_{1},\,\infty))$. 
As $\phi$ preserves time-orientation, we see that $\phi (\gamma_{i}) \subset \pi_{\R}^{-1}((T_{2},\,\infty))$. 
Corollary \ref{co: co_3} leads us to  
\begin{align*}
\dia (\pi_{T_{2}}(\phi (V_{i}))) 
= 
d_{T_{2}}(\pi_{T_{2}}(\phi (y_{i})),\,\pi_{T_{2}}(\phi (z_{i})))
\leq
2C^{\prime}. 
\end{align*}
Due to Lemma \ref{le: le_3.5}, 
we  obtain 
\begin{equation*}
\dia (\bigcup_{i=1}^{n(T_{1},\,\epsilon)} \pi_{T_{2}} (\phi (V_{i})) )
\leq 
 2 n(T_{1},\,\epsilon) C^{\prime}
< \dia (F_{T_{2}}). 
\end{equation*}
This contradicts the fact that  
$\displaystyle F_{T_{2}}=
\pi_{T_{2}}(\phi (F_{T_{1}}))=
\bigcup_{i=1}^{n(T_{1},\,\epsilon)} \pi_{T_{2}} (\phi (V_{i}))$.
\end{proof}
\section{Proof of Proposition~\ref{pr_11}.}\label{sec_11}
In this section, we show Proposition~\ref{pr_11}. 

Let $(M,\,h)$ be a Lorentzian manifold satisfying the assumptions of the proposition. 
First we prove that $(M,\,h)$ is realized as $(\R \times S,\, -dt^{2}+g_t)$, where $\{g_{t}\}_{t \in \R }$ is some 
smooth family of Riemannian metrics of $S$. 
Let $\pi :  NS \rightarrow S$ be the normal bundle over $S$. 
By  the timelike completeness, we can define the normal exponential map $\exp^{\perp}: NS \rightarrow M$.
\begin{lemma} 
The  map $\exp^{\perp}: NS \rightarrow M$ is a diffeomorphism. 
\end{lemma}
\begin{proof}
First, we show that $\exp^{\perp}$ is a local diffeomorphism. 
This reduces to proving that there exists no   focal point of $S$. 
Let $\gamma$ be a timelike geodesic ray starting at some point of $S$ perpendicular to $S$. 
It is sufficient to consider Jacobi fields orthogonal to the velocity of $\gamma$. 
Then the proof follows from the same argument as used 
in Hermann~\cite{MR0152969}.  

Next we show that $\exp^{\perp}$ is a bijection.  
For any  point of $M - S$, it is suffices to find  a unique geodesic from the point  perpendicular to $S$. 
Since such a geodesic exists due to Kim--Kim~\cite[Proposition~3.\,4.]{MR1269454}, 
all that we have to do is to prove the uniqueness. 
We may assume that $\exp^{\perp}(v_{0}) =\exp^{\perp}(v_{1})$, denoted by  $p$,  for some future directed 
vectors $v_{0},\, v_{1} \in NS$.   
Let  $T_{p}^{\past}M$ be the set consisting  of past directed 
timelike tangent  vectors of $T_{p}M$. 
Put $D(p) =\{v \in  T_{p}^{\past}M\,|\, \exp_{p}(v) \in S \}$. 
By the requirements of $S$, for any  $v \in T_{p}^{\past}M$  
there  is a unique positive number $ a$ such that 
$ a v \in D(p)$. 
Since the exponential maps  $\exp_{p}$ restricted to 
both $T_{p}^{\past}M$ and its boundary 
are transverse regular on $S$, 
$D(p)$ is a hypersurface in $T_{p}^{\past}M$,  whose boundary is  
the set of the lightlike vectors $v \in T_{p}M$ satisfying that $\exp_{p}(v) \in S$. 
The  curvature condition of $M$ implies that the map $\exp_{p} |_{T_{p}^{\past}M}$ is a local diffeomorphism 
(see \cite{On}). 
Here we  recall the definition of  $I_{S}(p)$ from  the assumption of the proposition. 
We see that the restricted map $\exp_p |_{D(p)} : D(p) \rightarrow I_{S}(p)$ is proper 
by the condition~$(\rm{iv})$ of the proposition. 
It follows that the map  $\exp_p |_{D(p)}$ is a covering map. 
We can take a geodesic $\tau$ connecting $\pi (v_{0})$ and $\pi(v_{1} )$  in $I_{S}(p)$ 
by using the condition $(\rm{iii})$. 
Let $\overline{\tau}$ be a lifting of $\tau$. 
We set $\gamma_{s}(u)=\exp_{p}(u \overline{\tau} (s) )$ for  $u \in [0,\, 1]$.  
The length $L^{-}(\gamma_{s})$ of the geodesic $\gamma_{s}  $ is given by 
\begin{equation*}
L^{-}(\gamma_{s})=\int_{0}^{1} \sqrt{-h (\Dot{\gamma_{s}}(u),
\,\Dot{\gamma_{s}}(u))}\ du. 
\end{equation*}
We define by $V_{s}(u)=\partial_{s}\gamma_{s}(u)$ the variation vector field $V_{s}$ along the geodesic $\gamma_{s}$. 
Let $V_{s}^{\perp}(u)$  be the component of 
$V_{s}(u)$ perpendicular to $\Dot{\gamma_{s}}(u)$.  
We should notice that $\nabla V_{s}^{\perp}(u)$ is spacelike. 
Here $R$ stands for the curvature tensor of $M$.  
We have 
\begin{align*}
&L^{-}(\gamma_{s}) \frac{d^{2} L^{-}(\gamma_{s})}{ds^{2}} \\
&= \int_{0}^{1}   h(R(V_{s}^{\perp}(u),\,\Dot{\gamma_{s}}(u))
\Dot{\gamma_{s}}(u),
\,V_{s}^{\perp}(u))  
 -  h(\nabla V_{s}^{\perp}(u),\,\nabla V_{s}^{\perp}(u) )\, du. 
\end{align*}
Hence $L^{-}(\gamma_{s})$ is concave. 
Since the ends of the domain are critical points of $L^{-}(\gamma_{s})$,  
 $L^{-}(\gamma_{s})$ is a constant function. 
We have $v_{0} = v_{1}$ as there is no focal point of $S$. 
\end{proof}
By time-orientability of $M$, 
there exists the normal  vector field $n:S \rightarrow NS$ with $h(n(p),\,n(p))=-1$. 
The map $\phi$ is given by $\psi :  \R \times S \ni (t,\,p) \mapsto t\, n(p) \in NS$. 
According to Gauss's lemma,  
we have $(\exp^{\perp} \circ \psi)^{*}h =-dt^{2} +g_{t}$  
where $t$ is the  parameter of $\R$, and $\{g_{t}\}_{t\in \R}$ is a smooth family of Riemannian metrics of $S$. 
By  abuse of notation, we write $h$ instead of $(\exp^{\perp} \circ \psi)^{*}h$.  

Next we show 
\begin{lemma}\label{le_new}
For any spacelike geodesic $\gamma$ in 
$\pi_{\R}^{-1}((0,\ \infty))$ (resp.\  $\pi_{\R}^{-1}((- \infty,\ 0))$),     
$d^{2} \pi_{\R} (\gamma (u))/ du^{2}$ is negative (resp.\ positive). 
\end{lemma}
\begin{proof}
Since $
d^{2} \pi_{\R} (\gamma (u))/ du^{2} =- 
 (\partial_{t}g)_{\gamma^0 (u)}
(\Dot{\gamma_{F}}(u),\,\Dot{\gamma_{F}}(u))/2$ as in the proof of Lemma~\ref{le: le_1}, 
we investigate the partial derivative of $g_{t}$ with respect to $t$. 
Take a point $x \in S$ and a  non-zero tangent vector $w \in T_{x}S $ arbitrarily. The curve $\gamma^{x}$ is defined by 
$\gamma^{x}(u)=(u,\,x) \in \R \times S$ for $u \in \R$. 
Then $\gamma^{x}  $ is a geodesic in $(\R \times S,\, h)$.  
Put $Y_{w}(u)= \partial_{s}\gamma^{c(s)}(u)|_{s=0}$,   
where $c$ is a curve $c:(-\epsilon ,\, \epsilon ) \rightarrow S $ 
such that $c(0)=x$, and that $\Dot{c}(0)=w$.  
 We see that $Y_{w}  $ is a Jacobi field along the geodesic 
$\gamma^{x}  $ such that $h(Y_{w}(u),\,Y_{w}(u))= g_{u}(w,\,w)$, and that $\nabla Y_{w} (u)$ is spacelike. 
We have 
$\partial_{t}^{2} h(Y_{w}(u),\,Y_{w}(u))  \geq \alpha^{2} h(Y_{w}(u),\,Y_{w}(u))$ by the curvature condition, 
where $\alpha$ is the constant  appearing in  the proposition.
 Since $S$ is totally geodesic,  $h(Y_{w}(0),\,\nabla Y_{w}(0))=0$. 
We obtain 
\begin{equation}\label{le: le_new}
\frac{\partial_{|u|} g_{u}(w,\,w)}{g_{u}(w,\,w)} 
=
\frac{\partial_{|u|} h(Y_{w}(u),\,Y_{w}(u))}{h(Y_{w}(u),\,Y_{w}(u))} \geq |\alpha \tanh (\alpha u)|.
\end{equation}
The lemma is proved. 
\end{proof}
\begin{remark}
The inequality (\ref{le: le_new}) indicates that $\{g_{t}\}_{t\in \R}$ satisfies $(\rm{H}$$)_{t_{0},c}$. 
\end{remark}

Finally we give a simple proof of the proposition without the main theorem. 
Take any isometry $\phi$ of $M$ and $T > 0$. 
Recall that $\pi_{\R}$ is the natural projection of $\R \times S$ onto $\R$,  and that 
$K_{T}= \pi_{\R}^{-1}([-T,\,T])$. 
Suppose that $\phi (S) \cap K_{T} = \emptyset$. 
It is suffices to  consider the case where $\phi (S) \subset \pi_{\R}^{-1}((T,\ \infty))$. 
For any $p \in \phi (S)$, 
let   $\{ e_{i}(p) \}_{i=1}^n$ be  an orthonormal basis  of $T_{p}\phi (S)$. 
As $\phi (S)$ is totally geodesic, we have    
\begin{equation}\label{eq: eq_last}
\int_{ \phi (S)} \Delta_{\phi (S)}(\pi_{\R}|_{\phi (S)}(p))\ dp= \int_{\phi(S)}\sum_{i=1}^{n}
(\Hess \pi_{\R}) (e_{i}(p),\,e_{i}(p))\ dp, 
\end{equation}
where $\Hess \pi_{\R}$ is the Hessian of $\pi_{\R}$ on $(\R \times S,\, h)$, 
$\Delta_{\phi (S)}$ is 
the Laplacian on $\phi (S)$.  
By Lemma~\ref{le_new} the right hand of the equality (\ref{eq: eq_last}) is negative. 
The divergence theorem implies that the left side of the equality (\ref{eq: eq_last})  is zero. 
This is a contradiction. 
Hence for any isometry $\phi$ of $M$ 
we have $\phi (S) \cap K_{T} \neq \emptyset$. 
The proof is complete.  \hfill $\square$ 
\subsection*{Acknowledgment}
The author wishes to express his thanks to  Masahiko Kanai 
for drawing his attention to the problem and for 
many valuable comments. 
He is also grateful to  Masanori Adachi for suggesting  Lemma~\ref{le: le_3.5}.
\bibliographystyle{amsplain} 
\bibliography{reference}
\nocite{
MR1299888, MR1617862, MR0424186, Beem, Borel, Ca, Kar, Kob_2, Kob_7, MR1217161, 
Kob, Kul_2, On, Wo, Wo_2, MR522040, MR652817, MR607014, MR0147566, MR0141672, 
MR0152969, MR1269454, MR0643822}

\end{document}